\documentclass[a4paper,11pt]{article}
\usepackage{stmaryrd}
\usepackage{amsfonts}
\usepackage{bbm}
\usepackage{graphics}
\usepackage{amscd}
\usepackage{mathrsfs}
\usepackage{latexsym,amssymb,amsmath,amscd,amscd,amsthm,amsxtra,xypic}
\usepackage[dvips]{graphicx}
\usepackage{dsfont}
\usepackage[all]{xy}

\usepackage{amscd, amssymb, amsmath, amsthm, graphics}
\usepackage{amsmath,amsfonts,amsthm,amssymb}
\usepackage{latexsym,amsmath}
\usepackage{graphicx,psfrag}
\usepackage{mathrsfs}

\newtheorem{thm}{Theorem}[section]
\newtheorem{lem}[thm]{Lemma}

\newtheorem{pro}[thm]{Proposition}

\newtheorem{rmk}[thm]{Remark}

\newcommand{\be }{\begin{eqnarray*}}
\newcommand{\ee }{\end{eqnarray*}}

\newcommand{\pf}{\noindent{\bf Proof.}\ }

\def\gpd{\,\lower1pt\hbox{$\longrightarrow$}\hskip-.24in\raise2pt
         \hbox{$\longrightarrow$}\,}

\def\qed{\hfill ~\vrule height6pt width6pt depth0pt}

\title
{ Fixed subgroups of automorphisms of\\ hyperbolic 3-manifold
groups}

\author{Jianfeng Lin and  Shicheng Wang}

\begin{document}

\maketitle 
\begin{abstract}

For fixed subgroups $Fix(\phi)$ of automorphisms $\phi$ on
hyperbolic 3-manifold groups $\pi_{1}(M)$, we observed that
$\text{rk}(Fix(\phi))<2\text{rk}(\pi_{1}(M))$ and the constant 2  in
the inequality is
 sharp;  we also classify all possible groups $Fix(\phi)$.

\end{abstract}

\section{Introduction}
 For a group $G$ and an automorphism $\phi: G \rightarrow G$, we
define $Fix(\phi)=\{\omega\in G|\phi(\omega)=\omega\}$, which is a
subgroup of $G$, and use $\text{rk}(G)$ to denote the rank of $G$.

The so called Scott conjecture proved 20 years ago in a celebrate work of M.
Bestvina and M. Handel \cite{BH} states that:

\begin{thm}\label{thm:scott conj}
For each automorphism $\phi$ on a free group $G=F_n$,
$$\text{rk}(Fix(\phi))\leq \text{rk}(G).$$
\end{thm}
In a recent paper by  B.J. Jiang, S. D. Wang and  Q. Zhang \cite{JWZ}, it
is proved that
\begin{thm}\label{thm:scott conj1}
For each automorphisms $\phi$ on a compact surface group
$G=\pi_{1}(S)$, $$\text{rk}(Fix(\phi))\leq \text{rk}(G).$$
\end{thm}

It is obvious that the bounds given in Theorem \ref{thm:scott conj}
and Theorem \ref{thm:scott conj1} are sharp and can be acheived by
the identity maps.

In this note, we will address the similar problem for hyeprbolic
3-manifold groups. We call a compact 3-manifold $M$ is {\it
hyperbolic}, if $M$ is orientable, and the interior of $M$ admits a
complete hyperbolic structure of finite volume (then $M$ is either
closed or $\partial M$ is a union of tori). Therefore $G=\pi_1(M)$
is isomorphic a cofinite volume torsion free Kleinian group. A  main
observation in this paper is the following

\begin{thm}\label{main}
For each automorphism $\phi$ on a hyperbolic 3-manifold group
$G=\pi_1(M)$, $$\text{rk}(Fix(\phi))<2\text{rk}(G),$$
and the upper
bound is sharp when $G$ runs over all hyperbolic 3-manifold groups.
\end{thm}

Theorem \ref{main} is a conclusion of the following Theorems
\ref{hyp counterex}, and \ref{sharp bound}.

\begin{thm}\label{hyp counterex}
There exist a sequences  automorphisms $\phi_{n}: \pi_{1}(M_{n})\to
\pi_{1}(M_{n})$ on closed hyperbolic 3-manifolds $M_{n}$ such that
$Fix(\phi_n)$ is the group of a closed surface, and
$$\frac{\text{rk}(Fix(\phi_n))}{\text{rk}(\pi_{1}(M_{n}))}>
2-\epsilon \text{ as } n\rightarrow\infty$$ for any $\epsilon >0$.
\end{thm}

\begin{thm}\label{sharp bound}
Suppose $\phi$ is an automorphism on $G=\pi_{1}(M)$, where $M$ is
a hyperbolic 3-manifold. Then $\text{rk}(Fix(\phi))<2\text{rk}(G)$.
\end{thm}

The proof of Theorem \ref{hyp counterex} is self-contained up to
some primary (and elegant) facts on hyperbolic geometry and on
combinatoric topology and group theory. Roughly speaking each ($M_i,
\phi_i$) in Theorem \ref{hyp counterex} is constructed as follows:
We first construct the hyperbolic 3-manifold $P_{i}$ with connected
totally geodesic boundary. Then we double two copies of $P_i$ along
their boundaries to get the closed hyperbolic 3-manifold $M_{i}$.
The reflection of $M_i$ alone $\partial P_{i}$ will induce an
automorphism $\phi_{i}:\pi_{1}(M_{i})\rightarrow\pi_{1}(M_{i})$ with
$Fix(\phi)=\pi_{1}(\partial P_{i})$. In this process all involved
ranks are  carefully controlled,  we get the inequality in Theorem
\ref{hyp counterex}.

To prove  Theorem \ref{sharp bound}, besides some combinatoric
arguments on topology and on group theory, we need the following
Theorem \ref{classifying theorem in torsion free} which classify all
possible groups $Fix(\phi)$ for automorphisms $\phi$ on hyperbolic
3-manifold groups. Recall that each automorphism $\phi$ on
$\pi_1(M)$ can be realized by an isometry $f$ on $M$ according to
Mostow rigidity theorem.

\begin{thm}\label{classifying theorem in torsion free}
Suppose $G=\pi_{1}(M)$, where $M$ is a hyperbolic 3-manifold, and
$\phi$ is a automorphism of G. Then $Fix(\phi)$ is one of the
following types: the whole group $G$; the trivial group
$\{e\}$; $\mathds{Z}$;  $\mathds{Z}\bigoplus\mathds{Z}$;  the
surfaces group $\pi_{1}(S)$, where $S$ can be  orientable
or not, and closed or not. More precisely
\begin{itemize}
\item[(1)] Suppose $\phi$ is induced by an orientation preserving isometry.
\item[(i)]  $Fix(\phi)$ is either $\mathds{Z}$, or $\mathds{Z\bigoplus\mathds{Z}}$, or $G$, or
$\{e\}$; moreover
\item[(ii)] if $M$ is closed, then $Fix(\phi)$ is either $\mathds{Z}$ or $G$;
\item[(2)] Suppose $\phi$ is induced by an orientation reversing isometry $f$.
\item[(i)] If $\phi^{2}\neq id$, then $Fix(\phi)$ is either $\mathds{Z}$ or $\{e\}$;
\item[(ii)] if $\phi^{2}=id$, then $Fix(\phi)$ is either $\{e\}$, or the surface group $\pi_1(S)$, where the surface $S$   is  pointwisely fixed by
$f$.
\end{itemize}
\end{thm}

Theorem \ref{classifying theorem in torsion free} is proved by using
the algebraic version Mostow Rigidity theorem, as well as some
hyperbolic geometry and   covering space argument.

The paper is organized as follows: In Section 2 we will prove
Theorem \ref{hyp counterex}, and we also generalize the examples
from closed hyperbolic 3-manifolds to hyperbolic 3-manifolds with
cusps. Theorem \ref{classifying theorem in torsion free} and Theorem
\ref{sharp bound} will be proved in Section 3 and Section 4
respectively.

Suppose a compact 3-manifold $M$ is hyperbolic and $S$ is a proper
embedded surface in $M$.  We say $S$ totally geodesic surface
implies that $S^o$, the interior of $S$, is totally geodesic, and
call $\partial S$ the boundary of $S^o$. Below we will use the same
$M $ ($S$) to present the interior of $M$ ($S$).

For terminologies not defined,  see \cite{He1} and \cite{Th1} for
geometry and topology of 3-manifolds, and  see \cite{SW}for group
theory.

 \bigskip\noindent\textbf{Acknowledgement}. The first author was partially supported
 by Beijing International Center of Mathematical Research.
    The second author was partially supported by grant
    No.11071006 of the National Natural Science Foundation of China.
    The authors  thank Ian Agol, David Gabai, Boju Jiang, and Hao Zheng
    for valuable communications  and suggestions.


\section{Proof of Theorem \ref{hyp counterex}}\label{Sec:example}

In this section, we construct examples stated in Theorem
\ref{hyp counterex}. Roughly speaking those examples are constructed
as follow: we first construct a hyperbolic 3-manifold $P$ with
totally geodesic boundary. Then we double it to get a closed hyperbolic 3 manifold $DP$. Now if we choose the base point
on the boundary of $P$, the reflection along $\partial P$ will
induce $\phi$ on the fundamental group of $DP$, and this automorphism $\phi$ will
have the property we desired.

There are different approaches to construct  hyperbolic manifolds
with totally geodesic boundaries. We will use the most original and the most direct one
due to Thurston. (For another  approach see Remark \ref{other appraoch}) .

In Thurston's Lecture Notes (Section 3.2 of \cite{Th1}), there is a
very concrete and beautiful construction of hyperbolic 3-manifolds
with totally geodesic boundaries involving primary  hyperbolic
geometry only.

In 3-dimensional hyperbolic space $H^{3}$,  there is a one-parameter
family of truncated hyperbolic tetrahedron as  in Figure 1: Each of
its 8 faces is totally geodesic; each of its 18 edges is geodesic
line segment. There are 4 triangle faces and 4 hexagon faces. The 12
edges of the 4 triangle faces have the same length, and  the remain
6 edges, we call them "inner edge", also have the same length. The
triangle faces are perpendicular to the hexagon faces. The angles
between hexagon faces are all equal and can be arbitrary angles
between $(0^{\circ},60^{\circ})$.

\begin{center}
\scalebox{0.5}{\includegraphics[0pt,0pt][363pt,326pt]{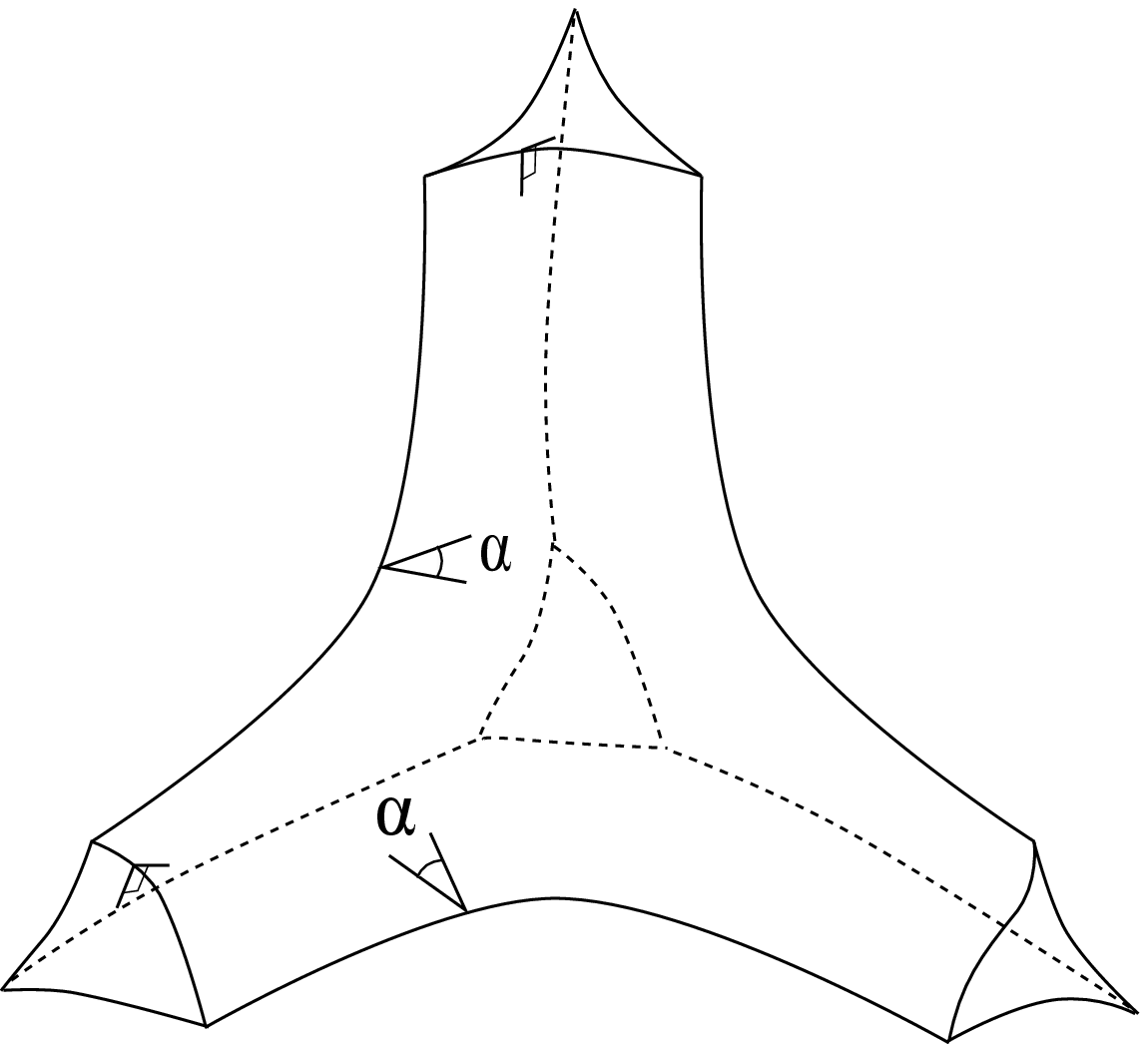}}

Figure 1
\end{center}

We will use those simplices to construct a hyperbolic manifold with
totally geodesic boundaries. Suppose we have some copies of
tetrahedron. We pair the faces of tetrahedron and gluing them
together (therefore some edges and vertexes are also  glued
together). After gluing, if we remove a neighborhood of the vertex,
we will get a topological manifold $P$. A tetrahedron with its
vertex neighborhood removed is homeomorphic to the truncated simplex
mention above. Suppose every $k$ edges of the tetrahedron are glued
together $ (k>6)$. We can set the face angle $\alpha$ of the
truncated simplex to be $\frac{2\pi}{k}$. Then the hyperbolic
structure of the truncated simplex fix together to give the
hyperbolic structure of $P$, and the triangle faces of the truncated
simplex are matched together to form the totally geodesic $\partial
P$.
It is easy to see that the number of vertex of tetrahedron (after gluing) equals the number of the boundary component. \\

Moreover, if we remove the neighborhood of the inner edges in $P$.
We will get a handlebody $H$. To see this, we remove the
neighborhood of the 6 edges of a tetrahedron. Topologically, it is
homeomorphic to $D^{3}$ and the 4 tetrahedron faces  are 4 disjoint
disks on $\partial D^{3}$. Then, we glue them together. If we glue
some 3 balls alone disks on their boundary, we get a handlebody. So
$P$ can be obtained by attaching $m$ 2-handles on a handlebody of
genus $n+1$. It is easy to see that $m$ is the number of inner edges
after gluing and $n$ is the number of tetrahedron.

Now we double $P$ along its boundary to get a closed hyperbolic
manifold $DP$. We have to control the rank of $\pi_{1}(DP)$. This is
done in the following lemma.

\begin{lem}\label{rank control}
Suppose $P$ is obtained by attaching $l$-handles to a handlebody of
genus $k$. Then $\text{rk}(\pi_{1}(DP))\leq k+l$($DP$ is the double
of $P$).
\end{lem}

\begin{proof}
Suppose $P$ is obtained from a handlebody  $H$ of genus $k$ by
attaching $l$ 2-handles $h_{1},h_{2}.....,h_{l}$ with attaching
curves $\gamma_{1},\gamma_{2},.....,\gamma_{l}$, where
$\gamma_{1},\gamma_{2},.....,\gamma_{l}$ are disjoint simple closed
curves on $\partial H$, and for each 2-handle $D^2\times I$,
$\partial D\times I$ is identified with the attaching  region
$N(\gamma_{i})$,  the regular neighborhood of $\gamma_{i}$, for some
$i$. Then we have

 $$P=H\bigcup_{\{ N(\gamma_i)\}} \{ h_{i}\}.$$

Note  in the doubling $DP$, the two copies of handlebody $H$ and
$H^{\prime}$ are glued together along $\partial H-\bigcup_{i}
N(\gamma_{i})$, and each two copies $D^{2}\times I$ of the 2-handle
$h_{i}$  are glued  alone the $D^{2}\times \partial I$ to get a
solid torus $S_{i}$, which is attached to $H\bigcup_{\partial H-\bigcup_{i}
N(\gamma_{i})} H^{\prime}$ along the torus boundary formed by two copies of
$N(\gamma_{i})$. So we have

$$DP=(H\bigcup_{\partial H-\bigcup_{i}
N(\gamma_{i})} H^{\prime})\bigcup_{i}S_{i}.$$ Because attaching
solid torus along the torus boundary  of $H\bigcup_{\partial
H-\bigcup_{i} N(\gamma_{i})} H^{\prime}$ does not increase the rank
of the fundamental group, we just need to control
$\text{rk}(\pi_{1}(H\bigcup_{\partial H-\bigcup_{i} N(\gamma_{i})}
H^{\prime}))$. We consider the two skeleton of this space. The two
skeleton of the handlebody $H$ consists of a surface of genus $k$
and $k$ copies of compressing disks. So the two-skeletons of
$H\bigcup_{\partial H-\bigcup_{i} N(\gamma_{i})} H^{\prime}$ can be
obtained as follows: starting from a surface $S_{k}$ of genus $k$,
we glue two copies of $S_{k}$ along $S_{k}-\bigcup_{i}
N(\gamma_{i})$. This is equally to attach $l$ copies of annulus
along $\partial N(\gamma_{i}), i=1,2,...,l$. Then we glue $k$
compressing disks on both side.

 Compared to the two-skeleton of $H$,
we see that only $l$ new generater are involved by the attached
annulus and the new attaching disk does not increase the rank of
fundamental group. So $\text{rk}(\pi_{1}(H\bigcup_{\partial
H-\bigcup_{i} N(\gamma_{i})} H^{\prime})) \leq k+l$. \end{proof}

\begin{center}
\scalebox{0.7}{\includegraphics*[0pt,0pt][451pt,294pt]{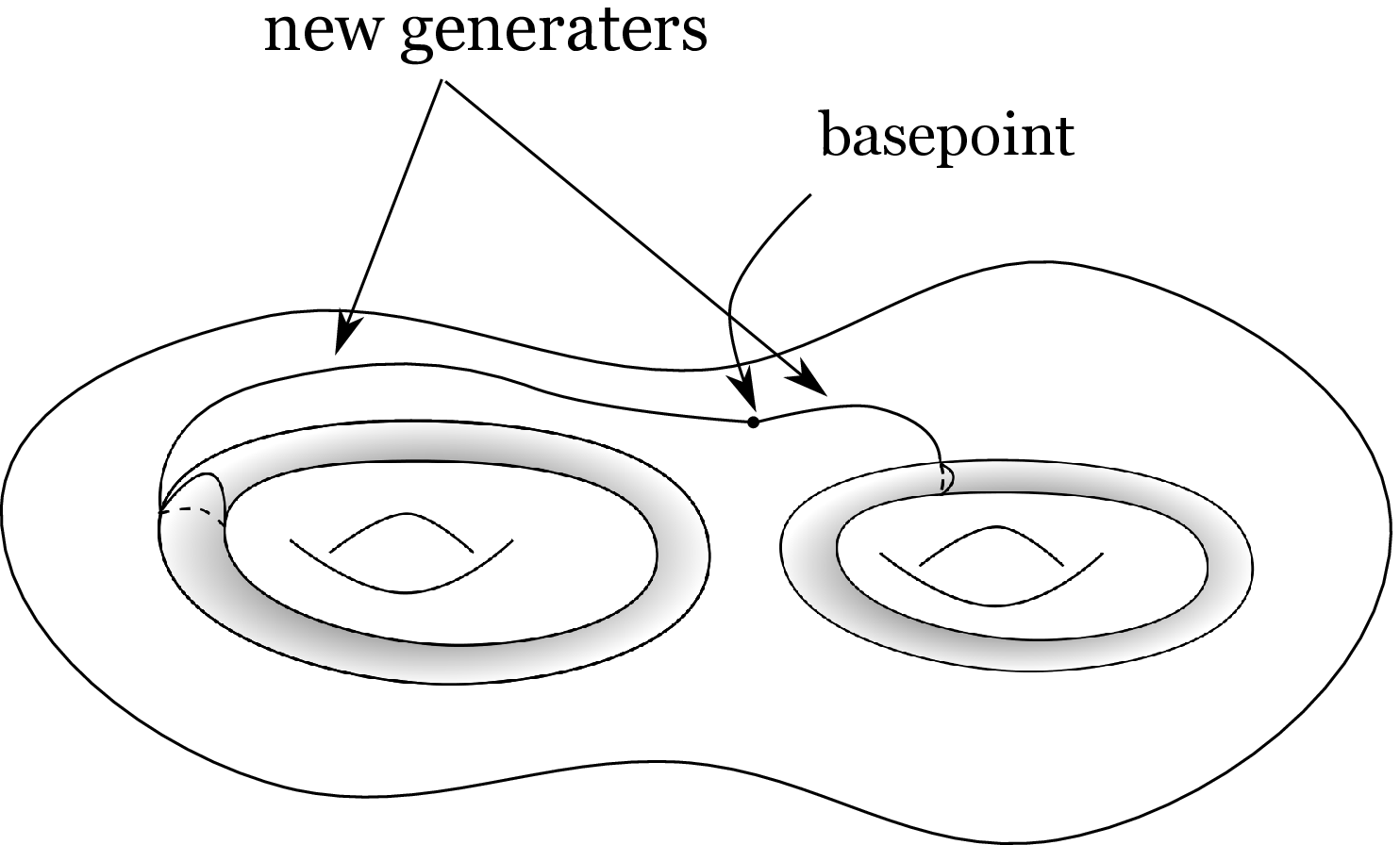}}

Figure 2
\end{center}

Now we can construct our examples. We start from $n$ ($n>3,3\nmid
n$) copies of the tetrahedron indicated in Figure 3, where the edges
are marked. We represent the faces by the edges around it. Each
tetrahedron $T _{i}$ has 4 faces
$(1,3,2)_{i},(4,5,3)_{i},(2,6,4)_{i},(5,1,6)_{i}$. Then we group the
$4n$ faces into 2n pairs:
$$[(1,3,2)_{i},(4,5,3)_{i+1}]; \,\,[(2,6,4)_{i},(5,1,6)_{i+1}],\,\,   i=1,2,....,n,\,
\text{ and } n+1\equiv1.$$

The two faces in each pair are glued together, and the orders of the
edges are preserved. (It's easy to see that the arrows on the edges are preserved too.) Then we get a simplex $X$.

\begin{center}
\scalebox{0.7}{\includegraphics*[0pt,0pt][211pt,222pt]{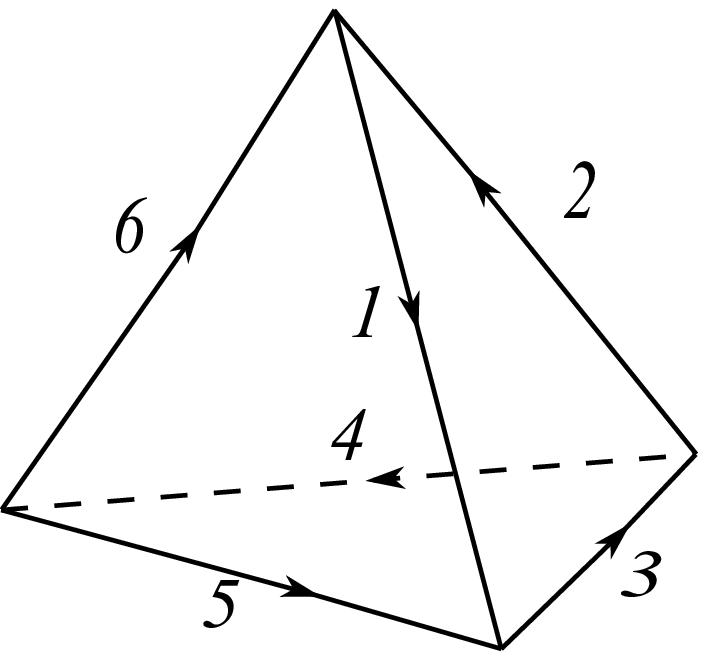}}

Figure 3
\end{center}

We write $a_{i}\leftrightarrow b_{j}$ to indicate that the edge $a$
in $T _{i}$ is glued together with  $b$ in $T _{j}$. With this
notation, we have

$$1_{k}\leftrightarrow 4_{k+1} \leftrightarrow 6_{k+2}\leftrightarrow
1_{k+3}; \,\, 2_{k}\leftrightarrow 5_{k+1} \leftrightarrow
3_{k}\leftrightarrow 2_{k-1}.$$

We first count the number of the edges after the gluing: Since
$3\nmid n$, we see that the $3n$ edges $1_{*},4_{*},6_{*}$ are glued
together, and the $3n$ edges $2_{*},5_{*},3_{*}$ are glued together,
so there are two edges in $X$.

Then  we count the number of the  vertices after the gluing: If we
denote the initial point and the terminal point of the directed edge
$i_k$
 by $I(i_k)$ and $E(i_k)$, then we have:
$$E(1_{k+1})\leftrightarrow I(3_{k+1})\leftrightarrow I(2_{k})\leftrightarrow I(4_{k})\leftrightarrow I(1_{k-1})\leftrightarrow E(2_{k-1}) \qquad (2.1)$$

The first, third and fifth identifications are shown in Figure 3.
The second and fourth identifications  follow from that respectively
$3_{k+1}$ and $2_{k}$, $4_{k}$ and $1_{k-1}$ are glued together as
direct edges. Since the two edges $[1_{*}],[2_{*}]$ after the
gluing. (2.1) implies that all ends of $[1_{*}],[2_{*}]$ are
identified to a point. Hence there is only one vertex in the simplex
$X$.

Finally we check the orientation: If we use the right hand
coordinate system, the face $(1,3,2)$ and $(5,1,6)$ correspond to
outward normal vectors while $(2,6,4)$ and $(4,5,3)$ correspond to
inward ones. Since each inward face is glued with an outward one,
the orientations are matched after gluing.

Now if we remove a regular neighborhood of the unique vertex, by the
discussion at the begin of this section, we get an orientable
hyperbolic three manifold $P$ with connected, totally geodesic
boundary.

As we have discussed,  $P$ can be constructed by attaching two
two-handles on a handle body of genus $n+1$. The genus of $\partial
P$ is $n-1$. Now double $P$ along its boundary to get a closed
3-manifold $DP$. If we choose a point $p\in \partial P$ as base
point, the reflection $f$ on $D(M)$ along $\partial P$ will induce
an automorphism $\phi:\pi_{1}(DP)\rightarrow\pi_{1}(DP)$.

\begin{lem}
In the construction above, $Fix(\phi)=Im(i_{*}(\pi_{1}(\partial
M)))$.
\end{lem}

\begin{proof} Note first that $P$ is boundary
imcompressible and $\pi_1(DP)$ is a free product of two copies of $\pi_1(M)$
amalgamated over their subgroup $\pi_1(\partial P)$, that is
$$\pi_1(DP)=\pi_1(M)_{\pi_1(\partial P)}*\pi_1(M),$$
and the amalgamation in induced from the doubling.

We will apply the standard form of elements in free product of groups with
amalgamations to prove the lemma. For convenient  we denote $\pi_1(M)_{\pi_1(\partial P)}*\pi_1(M)$ by  $G*_{H}*G'$, where
$G$ and $G'$ are two identical copies of $\pi_1(M)$ and $H$ is the $\pi_1(\partial P)$.
For each $g\in G$, denote $\phi(g)=g'$ (therefore $\phi(g')=g$) and clearly $\phi(h)=h$ for each $h\in H$.

 For each right coset $g_iH$ of $H$ in $G$, fix its representative $g_i$. We choose  the unit 1 as the representative for the right coset $H$ itself.
Then $\{g_i'H\}$ give the right coset decomposition of $H$ in $G'$ and fix representative $g'_i$ for  $g_i'H$.

According to  \cite[Theorem 1.7]{SW}, each  element $\gamma$ in  $G*_{H}*G'$ can be written uniquely in a form
$\gamma= a_1b'_1a_2b'_2...a_nb'_nh$, where $h\in H$, $a_{i}$ is some representative $g_j$ and $b_{i}$ is some representative $g_k$;
moreover $a_i=1$ implies $i=1$ and $b_i=1$ implies $i=n$. Then

$$\phi(\gamma)= a'_1b_1a'_2b_2...a'_nb_nh.$$

By uniqueness of the standard form, it is direct to see that if $\phi(\gamma)=\gamma$, then $\gamma=h$. hence only element in $H$
can be fixed by $\phi$.
\end{proof}

By Lemma 2.1, $\text{rk}(\pi_{1}DP)\leq n+3$. Since $\partial P$ has
genus $n-1$, by Lemma 2.2 $Fix(\phi)=Im(i_{*}(\pi_{1}(\partial
P)))\cong \pi_{1}(\partial P))$ has rank $2n-2$. For each $n>2$, construct such pair $(DP,\phi)$, and denoted as $(M_n,\phi_n)$. Then

$$ \frac{\text{rk}(Fix(\phi_n))}{\text{rk}(\pi_{1}(M_n))}\ge \frac{2n-2}{n+3}> 2-\epsilon, \text{ as }
n\rightarrow\infty$$ for any $\epsilon >0$. Hence we finished the
proof of Theorem \ref{hyp counterex}.\qed \vskip 0.5 truecm

The construction in Theorem \ref{hyp counterex} for closed
hyperbolic 3-manifold can be modified to the case of hyperbolic
3-manifold with cusps. Precisely

\begin{pro}\label{hyp counterex 2}
There exist a sequences of  hyperbolic  3-manifolds $M_{n}$ with
cusps and automorphisms $\phi_{n}:
\pi_{1}(M_{n})\rightarrow\pi_{1}(M_{n})$, such that $Fix(\phi_n)$
is a free group, and
$$\frac{\text{rk}(Fix(\phi_n))}{\text{rk}(\pi_{1}(M_{n}))}>
2-\epsilon \text{ as } n\rightarrow\infty.$$ for any $\epsilon >0$.
\end{pro}


We give some theorems to prove Proposition \ref{hyp counterex 2}.

\begin{thm}\cite{Th2}
Suppose $M$ is a hyperbolic 3 manifold with finite volume and $f$ is
a involution of $M$. Than $M$ admit a hyperbolic structure with
finite volume such that $f$ is an isometric with respect to this
structure.
\end{thm}

\begin{thm}\cite{Ko}\cite{Zh}
Suppose $M$ is a hyperbolic 3-manifold with finite volume and
$\alpha$ is a simple closed geodesic in $M$. Then $M-\alpha$ admit a
hyperbolic structure with finite volume.
\end{thm}

\begin{proof}  In the proof of Theorem \ref{hyp counterex}, the hyperbolic 3-manifold $P$ with connected totally
geodesic boundary is obtained by attaching  two 2-handles to a
handlebody $H$ along the attaching curve $\gamma_{1}$ and
$\gamma_{2}$. Now we choose a simple non-separating closed geodesic
$\alpha$ in $\partial P$ such that  $\alpha\subset
\partial H\setminus N(\gamma_{1})\cup N(\gamma_{2})$. Then $\alpha$ remains a geodesic in $D(P)$. Remove
$\alpha$ from the closed hyperbolic manifold $DP$, we get a new
hyperbolic manifold with a cusp by Theorem 2.5, denoted  by $DP'$.
The reflection $f$ on $D(P)$ along $\partial P$ defines a
restriction  on $DP'$, which is still an involution $f'$. By Theorem
2.6, $DP'$ admit a hyperbolic structure so that $f'$ is a isometry
under this hyperbolic structure. So as the fixed point set of an
isometry, the non-closed surface $\partial P-\alpha$ must be totally
geodesic, and therefore incompressible.

If we pick the base point on $\partial P-\alpha$ and consider the
automorphism $\phi$ induced by $f'$, the same combinational group
theory argument as before shows that $Fix(\phi)=\pi_{1}(\partial
P-\alpha)$. Because $\partial  P$ has genus $n-1$, $n$ is the same
as in the proof of Theorem \ref{hyp counterex}, $\partial P-\alpha$
is two punctured surface of genus $n-2$, hence $\pi_{1}(\partial
P-\alpha)$ is a free group of rank $2(n-2)+1=2n-3$.

Now we control $\text{rk}(\pi_{1}(DP'))$ via the same technique in
the proof of lemma 2.1: $DP'$ consist of two parts: the first part
is two copies of the handlebody $H$ glued  along $\partial
H\setminus N(\gamma_{1})\cup N(\gamma_{2})\cup N(\alpha)$; the
second part is two solid torus resulting from the doubling of the
2-handles. So the same argument as the proof of Lemma 2.1 shows that
$\text{rk}(\pi_{1}(DP'))<n+4$, and  we have

$$ \frac{\text{rk}(Fix(\phi))}{\text{rk}(\pi_{1}(DP'))}\ge \frac{2n-3}{n+4}> 2-\epsilon, \text{ as }
n\rightarrow\infty,$$ for any $\epsilon >0$.
\end{proof}

\begin{rmk}\label{other appraoch} There is another way to find hyperbolic 3-manifold with
totally geodesic boundary, which is based on a most profound result in the 3-manifold theory and a result on Heegaard splitting:

\begin{thm}\label{Geometrization}\cite{Th2}
Suppose $M$ is a compact 3-manifold $M$ with non-empty boundary and
infinite $\pi_1$. If $M$ contains no essential surface of genus
smaller than 2, Then $M$ admits a hyperbolic structure with totally
geodesic boundary.
\end{thm}


\begin{thm}\label{Heegaard distance}\cite{He}
Let $M$ be a closed oriented 3-manifold which is Seifert fibered or
which contains an essential torus. Then any splitting of $M$ is a
Heegaard distance $\leq2$ splitting.
\end{thm}

Theorem \ref{Heegaard distance} in \cite{He} is stated for  closed
3-manifolds, but the argument there can be used to prove the similar
theorem for non-closed case. Combine Theorem \ref{Heegaard distance}
and Theorem \ref{Geometrization} we can conclude that: If we attach
some two handles to a handlebody  so that the distances in curve
complex between the attaching curves of the two-handles and the
boundaries of the compressing disks of the handlebody are larger
than 3, then we will get   a hyperbolic 3 manifold with totally
geodesic boundary.

\end{rmk}

\section{Proof of  Theorem \ref{classifying theorem in torsion free}}
\noindent In this section, we will classify all the possible fixed
subgroups of automorphisms of cofinite volume klein groups.
We use $Iso \mathds{H}^{3}$ (resp. $Iso_+ \mathds{H}^{3}$) to denote the group
of (resp. orientation preserving) isometries of the 3-dimensional hyperbolic 3-space.

The most important tool is the following algebraic version of Mostow
rigidity theorem. Most topologists know the geometric version of
Mostow rigidity: Any homotopy equivalence between finite volume
hyperbolic 3 manifolds can be homotopied to an isometry. The
following algebraic version appears in \cite{MR}, which is
equivalent to the geometric version.

\begin{thm}\label{mostow rigidity}
Let $\Gamma_{1}$ and $\Gamma_{2}$ be two cofinite volume klein
groups, and $\phi:\Gamma_{1}\rightarrow\Gamma_{2}$ be an isomorphism
between them. Then there exist $\gamma\in Iso(\mathds{H}^{3})$
($\gamma$ may be orientation reversing) such that for any
$\alpha\in\Gamma_{1},\phi(\alpha)=\gamma\alpha\gamma^{-1}$.
\end{thm}

Now let's prove Theorem \ref{classifying theorem in torsion free}.

\begin{proof} Since $G=\pi_{1}(M)$ is a hyperbolic 3-manifold group,
  $G$ can considered as cofinite volume toriosn free Kleinian   group in $Iso_+ \mathds{H}^{3}$.
Now $G$ acts on $\mathds{H}^{3}$ as
 the deck  transformation group for the covering  $\pi :\mathds{H}^{3} \to \mathds{H}^{3}/G\cong M$. Then by
Theorem 3.1, there exist $\gamma\in Iso(\mathds{H}^{3})$ such that
for any $\alpha\in G,\phi(\alpha)=\gamma\alpha\gamma^{-1}$. Then

$$Fix(\phi)=\{\alpha\in G\mid \alpha\gamma=\gamma\alpha\}.\qquad(3.1)$$

Because $\gamma G\gamma^{-1}=G$, $\gamma$ induces  an isometry $f$
of $M$, such that the following diagram commutes.
$$
\xymatrix{
\mathds{H}^{3} \ar[rr]^{\gamma}\ar[d]_{\pi} &   &\mathds{H}^{3}\ar[d]_{\pi}  \\
M \ar[rr]^{f}  &   &M } \qquad (3.2)
$$

The  verification of Theorem \ref{classifying theorem in torsion
free} will be based on (3.1) and the verification will be divided
into two cases according to if $\gamma$ in (3.1) is orientation
preserving or not. \vskip 0.3 truecm

Case (1) $\gamma$ is orientation preserving.

(i) If $\gamma=e$ then clearly $Fix (\phi)=G$.

Below we assume that $\gamma$ is nontrivial. It is well-known that
each element in $G$ is either hyperbolic or parabolic;  moreover two
nontrivial elements $\alpha, \beta$ in $Iso_{+}(\mathds{H}^{3})$
commute if and only if in one of the following cases happen:

(a) Both $\alpha$ and $\beta$ are  parabolic elements and they share
the same fixed point in the infinite sphere $S^{\infty}$;

(b) Both $\alpha$ and $\beta$ are  non-parabolic elements (elliptic
or hyperbolic) and they share the same axis.

(c) Both $\alpha$ and $\beta$ are elliptic elements with rotation angle $\pi$ and their axis are perpendicular to each other.

Since elements in $G$ (therefore in $Fix(\phi)$) can not be elliptic, we just need to consider case(a) and case (b). In these two cases, if $\alpha,\beta,\gamma$ are all nontrivial, $\alpha$ commutes with $\beta$, $\beta$ commutes with $\gamma$, then $\alpha$ commute with $\gamma$. We see that $Fix(\phi)$ is a torsion free abelian group. As we know, the
fundamental group of a hyperbolic 3-manifold can contain torsion
free  abelian subgroups of ranks at most 2. So we have proved that:
if $\phi$ is induced by an orientation preserving map, $Fix(\phi)$
can only be $e$, $\mathds{Z}$, $\mathds{Z\bigoplus\mathds{Z}}$, or
$G$.

\vskip 0.2 truecm

(ii) If we further assume that $M$ is closed, then we have more
restrictions.

First $\pi_1(M)$ contains no subgroup $\mathds{Z\bigoplus}\mathds{Z}$ for a closed hyperbolic 3-manifold $M$.

Also we claim that $Fix(\phi)\neq \{e\}$. In fact the self-isometry
of a closed hyperbolic 3 manifold is always periodic. So there
exists positive integer $n$ such that  $f^{n}=id$. By
commuting diagram (3.2), $\gamma^{n}$ induces the identity on $M$, therefore  $\gamma^{n}\in G$. Then clearly $\gamma^{n}\in
Fix(\phi)$. If $\gamma^{n}\neq e$, then $Fix(\phi)$ is not trivial.
If $\gamma^{n}= e$, then $\gamma$ is an elliptic element, so there
is an axis $l$ pointwise fixed by $\gamma$. By the commuting
diagram (3.2), $\pi(l)$ is pointwise fixed by the isometry $f$. Since
$M$ is closed, the fixed point set of a orientation preserving
isometry can only be closed geodesics. So $\pi(l)$ is a closed
geodesic. This means that there is a hyperbolic covering
transformation $\alpha\in G$ sharing the axis $l$ with $\gamma$, so
$\alpha\gamma=\gamma\alpha$, and therefore $\alpha\in Fix(\phi)$ by (3.1).

We have actually proved that if $\phi$ is induced by a orientation
preserving map and $M$ is closed. Then $Fix(\phi)$ can  be either
$\mathds{Z}$ or $G$.

\vskip 0.3 truecm

Case (2) $\gamma$ is orientation reversing.

Note $Fix(\phi)\subseteq Fix(\phi^{2})$ and  $\phi^{2}$ is induced
by an orientation preserving map. There are two subcases now:

(i) $\phi^{2}\neq id$. Then $Fix(\phi^2)$ can only be $e,\mathds{Z}$
or $\mathds{Z\bigoplus\mathds{Z}}$ by Case (1) (i) and its proof, therefore $Fix(\phi)$ can only be $e,\mathds{Z}$
or $\mathds{Z\bigoplus\mathds{Z}}$.

But in fact the situation $\mathds{Z\bigoplus\mathds{Z}}$ never
happens. Because in this situation
$Fix(\phi^{2})=\mathds{Z\bigoplus\mathds{Z}}$, which was generated
by two parabolic elements $\beta_1$ and $\beta_2$ sharing the same
fixed point $p$ on the infinite sphere. Since we assume that
$\phi^{2}\neq id$, $\gamma^{2}$ is also a parabolic element  with
the the fixed point $p$. Since $Fix (\gamma)\subset Fix (\gamma^2)$,
one can derived that $\gamma$ has the unique fixed point  $p$ in
$\mathds{H}^{3}\cup S^\infty$, and in the upper-half model of
$\mathds H^3$ (we set $p=\infty$), $\gamma^{2},\beta_{1},\beta_{2}$
are translations along some directions $v_{1},v_{2},v'$
respectively.

Consider their extended action on the plane $z=0$. Then $\gamma$
acts as a conformal (orientation reversing) map on this plane (in
order to see this, we can just compose $\gamma$ with an arbitrary
reflection $r'$ which fixes $p$ to get an orientation preserving
isometry. By classical fact, both $r'\circ\gamma$ and $r'$ act
conformally on plane $z=0$, then so does $\gamma$). Because
$\gamma^{2}$ is a translation, $\gamma$ must act as an orientation
reversing isometry on the plane $z=0$. So $\gamma\mid_{z=0}=r\circ
h$, where $h$ is a translation along the direction $v'$ and $r$ a
reflection along an invariant line of $h$. Then it is a direct
verification that $\gamma$ commutes with $\beta_i$ if and only if
the directions of $v'$ and $v_i$ are either the same or opposite.
But the directions of $v_1$ and $v_2$ are neither the same nor
opposite, so $\gamma$ can not commute with both two generators of
$\mathds{Z\bigoplus\mathds{Z}}$.



We have proved that in this subcase $Fix(\phi)$  is either  $e$ or $\mathds{Z}$.

\vskip 0.2 truecm

(ii) $\phi^{2}=id$. Then $\gamma^{2}$ commute with the whole group
$G$. So $\gamma^{2}=e$, $\gamma$ has order 2. An order 2 orientation
reversing isometry of $H^{3}$ can only be the reflection along a
single point or reflection along a geodesic plane.

If $\gamma$ is the reflection along a point $p\in\mathds{H}^{3}$,
then $p$ is the only fixed point of $\gamma$. For any $\alpha\in
Fix(\phi)$, we have $\gamma\alpha=\alpha\gamma$ by (3.1). Hence
$\gamma\alpha(p)=\alpha\gamma(p)=\alpha(p)$, that is  $\gamma$ also
has fixed point $\alpha(p)$, hence $p=\alpha(p)$. Because $\alpha$
is a covering transformation, we msut have $\alpha=e$.  We have
proved that if $\gamma$ is a reflection along a single point, then
$Fix(\phi)$ is trivial.

If $\gamma$ is the reflection along a totally geodesic plane $P$.
Then $P$ is pointwise fixed by $\gamma$. Because of the commuting
diagram (3.2), $\pi(P)$ is pointwise fixed by $f$. We know that the
fixed point set of an orientation reversing isometry of a hyperbolic
3-manifold must be totally geodesic surfaces if it is dimension 2.
So $\pi(P) \simeq S$, $S$ is a totally geodesic surface in $M$. ($S$
may be non-orientable although $M$ is orientable. And if $M$ has
cusps, $S$ may have cusps too).

For any $\alpha\in Fix(\phi)$, $\gamma\alpha=\alpha\gamma$. Then for
each $x\in P$,  $\gamma\alpha(x)=\alpha\gamma(x)=\alpha(x)$, that is
$\alpha(x)\in P$. It follows that  $P$ is invariant under $\alpha$.

Conversely, suppose a covering transformation $\alpha\in G$ such
that $\alpha(P)=P$. Then it is easy to see that
$\gamma\alpha=\alpha\gamma$ and therefore   $\alpha \in Fix(\phi)$.
So $Fix(\phi)$ is exactly the covering transformations of the
universal covering map $P\xrightarrow[]{\pi|P}S$. We have proved
that $Fix(\phi)\cong\pi_{1}(S)$. \end{proof}


\section{Proof of Theorem \ref{sharp bound}}

By Theorem \ref{classifying theorem in torsion free}, the situation
$\text{rk}(Fix(\phi))> \text{rk}(\pi_1(M))$ can appear only if Case
(2) (ii) in Theorem \ref{classifying theorem in torsion free}
happens, and if Case (2) (ii) in Theorem \ref{classifying theorem in
torsion free} happens, then $Fix(\phi)=\pi_1(S)$ for some surface
$S$ which is pointwise fixed by an orientation reversing isometry
$f$ of order 2 on $M$. So the proof of the Theorem \ref{sharp bound}
will be completed by the following

\begin{pro}\label{surface rank}
Suppose $M$ is a hyperbolic 3-manifold and  $S$ is a proper embedded
surface in $M$.  If there is an orientation reversing isometry $f$
of order 2 on $M$ fixing $S$ pointwisly. Then
$$\text{rk}(\pi_{1}(S))<2\text{rk}(\pi_{1}(M)).$$
\end{pro}

To prove Proposition \ref{surface rank}, we need the following lemma
which contains several elmentary facts:
\begin{lem}\label{known facts}
(1) Suppose $G$ is a group with subgroup $H$ of index n. Then
$$\text{rk}(G)\geq \frac{\text{rk}(H)+n-1}{n}.$$

(2) Suppose $S$ is a boundary component of the compact 3-manifold
$M$, and $D_S(M)$ is the doubling of two copies of $M$ along $S$.
Then $$\text{rk}(\pi_{1}(D_S(M)))\geq \text{rk}(\pi_{1}(M)).$$

(3) (Half die half alive Lemma) Suppose $M$ is a  compact orientable
3-manifold. Then $$ \text{dim}\{\text{image } i^*:H^{1}(\partial
M,\mathds{Q})\to
H^{1}(M,\mathds{Q})\}=\frac{\text{dim}H^{1}(\partial
M,\mathds{Q})}2\qquad$$ where $i^*$ is induced by the inclusion $i:
\partial M\to M$.

(4) Suppose $M$ is a  compact orientable 3-manifold and $S$ is an
incompressible boundary component of $M$. If the homomorphism
$\pi_1(S)\to \pi_1(M)$ induced by the inclusion is not a surjection,
then there is a finite covering $p:\tilde M\to M$  such that
$p^{-1}(S)$ contains more then one component.

\end{lem}

\pf (1) If $\text{rk}(G)=k$, we can find a 2-dimensional CW complex
$X$ with fundamental group $G$ so that in $X$ has only one vertex
and $k$ edges. Let $\tilde X$ be the n-sheet covering space
corresponding to the subgroup $H$. Then there $n$ vertex and $nk$
edge in $\tilde X$, and this lifted CW complex of $\tilde X$
provides a presentation of  $H$ with $n(k-1)+1$ generators. So
$\text{rk}(H)\leq n(k-1)+1$, that is
$$\text{rk}(G)=k\ge \frac{\text{rk}(H)+n-1}{n}.$$

(2) It is clear that there is a reflection $f$ about $S$ on
$D_S(M)$,  which provides a folding map $D_S(M)\rightarrow
D_S(M)/f\cong M$, and  which is obviously induce a epimorphism
between fundamental groups. Hence

$$\text{rk}(\pi_{1}(M))\geq \text{rk}(\pi_{1}(M_{1})).$$

(3) See \cite[Section 23]{Mo}

(4) By the assumptions,  $\pi_1(S)$ is a proper subgroup of
$\pi_1(M)$ (we pick a base point $x$ on $S$ for both $\pi_1(S)$ and
$\pi_1(M)$). Pick an non-zero element $\alpha$ of $\pi_1(M)$ but not
in $\pi_1(S)$. By \cite[Theorem 1]{LN} (peripheral subgroups are
separable), there is a finite index subgroup $H$ of $\pi_1(M)$ which
contains $\pi_1(S)$ but does not contain $\alpha$. Consider the
finite covering $p:\tilde M\to M$ corresponding to $H$, then there
is a component $\tilde S$ of $p^{-1}(S)$ homeomorphic to $S$ (since
$\pi_1(S)\subset H$), and $p^{-1}(S)$ has more than one component
(since $\alpha$ does not in $H$, the lift $\tilde \alpha$ of
$\alpha$ with one end in $\tilde S$ must has another end in another
component of  $p^{-1}(S)$). \qed

\begin{lem}\label{surface rank1}
Suppose $M$ is a  compact orientable 3-manifold and $S$ is an
incompressible boundary component of $M$ with genus $g$. If the
homomorphism $\pi_1(S)\to \pi_1(M)$ induced by the inclusion is not
a surjection, then $\text{rk}(\pi_{1}(M))>g$.

\end{lem}
\pf Let $p: \tilde M\to M$ be the  $n$ sheet covering provided by
the proof of Lemma \ref{known facts} (4).  Then the  preimage of $S$
has $m>1$ component. Now we can compute the sum of the genus of
these $m$ boundary components. Because $S$ has euler number $2-2g$,
the sum of the euler number of the preimage of $S$ is $n(2-2g)$.
Therefore, the sum of their genus is $n(g-1)+m$.  Since
$H^{1}(\partial M',\mathds{Q})$ is a direct sum of the homology of
the boundary components, it contains at least $2n(g-1)+2m$ copies of
$\mathds{Q}$.

By Lemma \ref{known facts} (3), we have
$\text{rk}(H^{1}(M',\mathds{Q}))\geq n(g-1)+m$. Since
$H^{1}(M',\mathds{Q})$  is a quotient group of $\pi_{1}(M')$, we
have $\text{rk}(\pi_{1}(M'))\geq n(g-1)+m$.  Since $\pi_{1}(M')$ is
subgroup of $\pi_{1}(M)$ of  index $n$, by Lemma \ref{surface rank}
(1) we have
$$\text{rk}(\pi_{1}(M))\geq\frac{n(g-1)+m+n-1}{n}= g+\frac{m-1}{n}.$$

 Since $m>1$. Then Lemma
\ref{surface rank1} is proved.
 \qed

\begin{rmk}\label{stallings}
Suppose $M$ is a hyperbolic 3-manifold and  $S$ is a closed embedded
surface in $M$ which is pointwisely fixed by an orientation
reversing involution on $M$. Cutting $M$ along $S$, we get a compact
3-manifold $M'$ (may be not connected) with a boundary component
$S$. In the proof of Theorem \ref{surface rank} we will apply Lemma
\ref{surface rank1} to $S\subset M'$ directly by the following
reason:

Using the fact that $S$ is a proper embedded surface in $M$ which is
pointwisely fixed by an orientation reversing involution $f$ on $M$
and
 $M$ contain no
essential spheres and essential tori, it follows that (1) $S$ must
be incompressible (otherwise $M$ would contain essential spheres);
(2) the homomorphism $\pi_1(S)\to \pi_1(M')$ induced by the
inclusion is not a surjection, (otherwise $M'=S\times [0,1]$ by a
result of J. Stallings \cite[10.2 Theorem]{He1}, and then $M$ would
contain essential tori).
\end{rmk}
\vskip 0.3 true cm

Now we start to prove Proposition  \ref{surface rank}.

 \begin{proof} There are several cases to be
considered. The surface may be either separating or not, either
orientable or not, either closed or not. The proofs of all those
cases are similar, but some subtle differences may appear. So we
write all the details for each case.

Suppose the surface  $S$ has $k$ boundary components,  denoted by
$c_{1},c_{2},....c_{k}$. Different $c_{i}$ may be contained in the
same torus component, but a torus component can contain at most $2$
of such $c_{i}$. In fact, suppose $c_{i},c_{j}\subset T$, a torus
component  of $\partial M$. Now $f|T$ is an orientation reversing
involution on $T$. $c_{i}$ and  $c_{j}$ are two parallel circles on
$T$, fixed pointwisely by $f$. It's easy to see that $f$ interchange
the two connected components of $T-c_{i}\bigcup c_{j}$, so $T$ can
not contain any other component $c_l$ other than $c_{i}$ and
$c_{j}$.

 Without loss of
generality, we can assume that among the boundary components of $S$.
$c_{2i-1}$ and $c_{2i}$ are in the same torus boundary component,
$i=1,2,....m$. and $c_{2m+j}, j=1,2,...,l$ are contained in other
$l$ different boundary components.

Case (1) $S$ is orientable surface of genus $g$. Note
$$\text{rk}(\pi_{1}(S))=2g+k-1 \text{ if } k>0 \text{ and }
\text{rk}(\pi_{1}(S))=2g \text { if } k=0.\qquad (4.1)$$

We will divided the discussion into two  subcases according to if
$S$ is separating or not.

(i) $S$ is separating.  Then it is easy to see that $k=2m$ and
$l=0$. Cutting $M$ along the surface $S$, we get two homeomorphic
components $M_{1}$, $M_{2}$, and $f$ interchanges them. Now each
pair $c_{2i-1}, c_{2i}$ bounds an annulus in $M_1$ connecting $S$,
which increses the  genus of $S$ by 1. So we obtained a  boundary
component of $M_{1}$ with genus $(g+\frac{k}{2})$.

If $k>0$, then
$$2\text{rk}(\pi_{1}(M))\geq2\text{rk}(\pi_{1}(M_{1}))\geq
2\text{rk}(H_{1}(M_{1},Q))\geq
2(g+\frac{k}{2})>2g+k-1=\text{rk}(\pi_{1}(S)).$$  The first and  the
third inequalities and the last equality are based on Lemma \ref{known facts} (2), (3)
and (4.1) respectively.

If $k=0$, then $M_1$ is a hyperbolic 3-manifold with a totally
geodesic boundary component  $S$, which is incompressible. Then
$$2\text{rk}(\pi_1(M))\geq 2\text{rk}(\pi_1(M_1))>2g=\text{rk}(\pi_1(S)).$$
Those two inequalities and one equality are based on Lemma
\ref{known facts} (2), Lemma \ref{surface rank1} (also Remark
\ref{stallings}) and (4.1) respectively.

 (ii)  $S$ is non-separating.  Cutting $M$ along $S$ we get a new
connected manifold $M'$ with two copies of $S$, denoted by $S_{1}$
and $S_{2}$, in $\partial M'$.

Suppose $S_{1}$, $S_{2}$ are contained in the same boundary component $S'$ of $M$. Then $k>0$ and $S'$ consist of $S_{1}$, $S_{2}$ and $2m+l$ annulus, which is clearly closed and orientable, and $$g(\partial M')\ge g(S')=2g+2m+l-1=2g+k-1\qquad (4.2).$$

As before,
by Lemma \ref{known facts} (2), (3) and (4.2) we get
$$\text{rk}(\pi_{1}(M'))\geq \text{rk}(H_{1}(M',Q))\geq
\frac{\text{rk}(H_{1}(S',Q))}{2}=2g+k-1\qquad (4.3).$$ Now
$f'=f|_{M-S}$ is an involution on $M'$, which
 keeps the boundary component $S'$ invariant and interchanges $S_{1}$ and
$S_{1}$. Now let's take two copies of $M'$, denote them by $M'_{1}$
and $M'_{2}$, and glue them on $S_{1}$ and $S_{2}$ via the identity
 to get a new manifold $\widetilde{M}$.
Let $r$ be the reflection on $\tilde M$ about $S_1\cup S_2$. Then we
have a free involution $\tilde f$ on $\tilde M$ defined as

$$\tilde f|M'_1=f'\circ r \text{ and  } \tilde f|M'_2=r\circ f'.$$

 It is easy to verify
 that
$\pi: \widetilde{M}\to \widetilde{M}/\tilde f= M$ is a two fold
covering.

Applying Lemma \ref{known facts} (1) for $n=2$,  Lemma \ref{known
facts} (2),  (4.3) and (4.1) (recall that in this case $k>0$), we
have
$$2\text{rk}(\pi_{1}(M))\geq\text{rk}(\pi_{1}(\widetilde{M}))+1\geq \text{rk}(\pi_{1}(M'))+1\geq2g+k>\text{rk}(\pi_{1}(S)).$$

Suppose $S_{1}$ and $S_{2}$ belong to  two different components of
$\partial M'$. Then $l=0$ and each component consists of one $S_i$
and $m$ annuli,  hence

$$g(\partial M')\ge g(S_1)+m+g(S_2)+m=2g+k\qquad (4.4).$$

 Doubling two copies of $M'$ along $S_{1}$ and $S_{2}$ and
constructing a 2-fold covering $\tilde M\to M$ using the involution
$f$ as before, apply (4.4) we can prove similarly that:
$$2\text{rk}(\pi_{1}(M))\geq2g+k+1>\text{rk}(\pi_{1}(S)).$$

(2) $S$ is non-orientable surface of genus $g$ (connected sum of $g$
real projective planes).  Note
$$\text{rk}(\pi_{1}(S))=g+k-1 \text{ if } k>0 \text{ and }
\text{rk}(\pi_{1}(S))=g \text { if } k=0.\qquad (4.5)$$

In this case $S$ is non-separating. As before, we cut $M$ along $S$
to get a new manifold $M'$ with one boundary component $S'$ consisting of
the orientable double cover $\widetilde{S}$ of $S$ and $2m+l$
annulus, and we have $$g(\partial M')\ge g(S')=g-1+2m+l=k+g-1\qquad (4.6)$$

 The involution $f'=f|_{M-S}$ provides a covering transformation of
$\widetilde{S}\rightarrow S$. Again, we glue two copies of $M'$
along $\widetilde{S}$ to get the manifold $\widetilde{M}$ and a
double covering  $\widetilde{M}\to M$.

If $k>0$, as before, applying Lemma \ref{known facts} (1) for $n=2$,
Lemma \ref{known facts} (2), (3),   (4.6) and (4.5) in order, we
have
$$2\text{rk}(\pi_{1}(M))\geq \text{rk}(\pi_{1}(\widetilde{M}))+1\geq
\text{rk}(\pi_{1}(M'))+1\geq \text{rk}(H_{1}(M',Q))+1 \geq k+g
> \text{rk}(\pi_{1}(S)).$$

If $k=0$, applying Lemma \ref{known facts} (1) for $n=2$, Lemma
\ref{known facts} (2), Lemma \ref{surface rank1} (also Remark
\ref{stallings}), (4.6)  and (4.5) in order, we have
$$2\text{rk}(\pi_{1}(M))\geq \text{rk}(\pi_{1}(\widetilde{M}))+1\geq
\text{rk}(\pi_{1}(M'))+1>g-1+1=\text{rk}(\pi_{1}(S)).$$

We finished the proof of Proposition \ref{surface rank}.\end{proof}

\vskip 0.5 truecm

Jianfeng Lin

Department of Mathematics, Peking University, Beijing, China

linjian5477$@$pku.edu.cn

\vskip 0.5 truecm

Shicheng Wang

Department of Mathematics, Peking University, Beijing, China

wangsc$@$math.pku.edu.cn

\end{document}